\documentclass[12pt]{article}
\usepackage{psfrag,amsmath,amssymb,latexsym,epic, eepicemu, epsfig,
 float, enumerate}
\usepackage{amscd}
\usepackage{amsfonts}
\usepackage{graphicx}
\usepackage{epsfig,psfrag,amsmath,amssymb,latexsym}
\usepackage{amscd}
\usepackage[dvips]{color}
\usepackage{amsfonts}
\usepackage{graphics}
\usepackage{graphicx}
\usepackage{float}
\usepackage[english]{babel}
\usepackage[T1]{fontenc}
\usepackage[latin1]{inputenc}
\newtheorem{theorem}{Theorem}[section]

\newtheorem{corollary}{Corollary}[section]

\newtheorem{lemma}{Lemma}[section]

\newtheorem{proposition}{Proposition}[section]

\newenvironment{proof}[1][Proof]{\textbf{#1.} }{\ \rule{0.5em}{0.5em}}

\def\t{\theta}
\def\tt{\hat{\theta}^n}
\def\f{{\hat f}_n}

\begin{document}

\textbf{
\Large{\centerline{Strong consistency of pseudo-likelihood parameter}
    \centerline{estimator for univariate Gaussian mixture models}}}
%
%\textbf{\centerline{(Running title: Consistency of pseudo-%likelihood estimator)}}
%
%
%
\newline \newline
\centerline{J\"{u}ri Lember$^{a\ast}$, jyril@ut.ee}
\centerline{Raul Kangro$^a$, raul.kangro@ut.ee}
\centerline{Kristi Kuljus$^a$, kristi.kuljus@ut.ee}
\newline \newline
\centerline{\small{$^a$Institute of Mathematics and Statistics, University of Tartu;}}
\centerline{\small{Narva mnt 18, 51009, Tartu, Estonia}}
\centerline{\small{$^{\ast}$ Corresponding author}}

\begin{abstract} We consider a new method for estimating the parameters of univariate Gaussian mixture models. The method relies on a nonparametric density estimator $\hat{f}_n$ (typically a kernel estimator). For every set of Gaussian mixture components, $\hat{f}_n$ is used to find the best set of mixture weights. That set is obtained by minimizing the $L_2$ distance between $\hat{f}_n$ and the Gaussian mixture density with the given component parameters. The densities together with the obtained weights are then plugged in to the likelihood function, resulting in the so-called pseudo-likelihood function. The final parameter estimators are the parameter values that maximize the pseudo-likelihood function together with the corresponding weights. The advantages of the pseudo-likelihood over the full likelihood are: 1) its arguments are the means and variances only, mixture weights are also functions of the means and variances; 2) unlike the likelihood function, it is always bounded above. Thus, the maximizer of the pseudo-likelihood function -- referred to as the pseudo-likelihood estimator -- always exists. In this article, we prove that the pseudo-likelihood estimator is strongly consistent.

\end{abstract}

\section{Introduction}
\subsection{Pseudo-likelihood estimator}
We consider the problem of parameter estimation in a univariate Gaussian mixture model with $k$ components. In \cite{A1}, a new method called the {\it pseudo-likelihood approach} was proposed for estimating all parameters (means, variances, weights). In the present article, we establish the strong consistency of the pseudo-likelihood estimator. The pseudo-likelihood approach relies on a nonparametric density estimator $\hat{f}_n$ (typically a kernel estimator), which is used to estimate the mixture weights. More precisely, letting $\theta_i=(\mu_i,\sigma_i)$ denote the mean and variance of the $i$-th mixture component, and $g(\t_i,\cdot)$ the corresponding Gaussian density, the weights are obtained by minimizing the $L_2$ distance between $\hat{f}_n$ and the mixture density with fixed parameters $\theta=(\t_1,\ldots,\t_k)$:
\begin{align}\label{weights0}
v^n(\theta):=\arg\inf_{w\in S_k}\|\hat{f}_n(\cdot)-\sum_{i=1}^k w_i g(\t_i,\cdot)\|.
\end{align}
Here, $S_k$ denotes the $(k-1)$-dimensional simplex, defined as $$S_k := \{(w_1, \ldots, w_k): w_i \geq 0,\ \sum_i w_i = 1\},$$ and $\|\cdot\|$ denotes the $L_2$ norm. The second step of the pseudo-likelihood  approach is to plug the obtained weights $v^n(\t)$, together with the parameters $\t$, into the likelihood function to get the {\it pseudo-likelihood function}
$$L_n(\t):=\prod_{t=1}^n\big(\sum_{i=1}^k v^n_i(\t)g(\t_i,y_t)\big),$$
where $y_1,\ldots,y_n$ is the observed sample. In \cite{A1}, it was proved that for distinct $y_1,\ldots,y_n$, $L_n(\t)$ is bounded even when the variances are not bounded away from 0 (Theorem 2.2 in \cite{A1}). This implies that the maximizer $\hat{\theta}_n$ of $L_n(\t)$ exists almost surely. The estimator
$\hat{\theta}_n$ will be referred to as the {\it maximum pseudo-likelihood estimator}. The main goal of the present article is to show that, under an i.i.d.~sample, the maximum pseudo-likelihood estimator is strongly consistent: $\hat{\theta}_n\stackrel{a.s.}{\to} \t^*$ and $v^n(\hat{\theta}_n)\stackrel{a.s.}{\to} w^*$, where $\t^*$ and $w^*$ denote parameters and weights of the true distribution. In the sequel, the term "parameters" refers to the means and variances, excluding the weights -- although, strictly speaking, the weights are also parameters of the mixture density. The consistency result is stated as Theorem \ref{thm}.

Estimating the mixture weights using the $L_2$ distance has already been (implicitly) exploited in the so-called DUDE method for signal denoising \cite{DUDE05a, DUDE05b}. The setting and objective in \cite{DUDE05a, DUDE05b} differ somewhat from ours. In terms of the present paper, their case corresponds to that with known component densities $g_i$ and the weights $w$ are   estimated as
$\hat{w}=A^{-1}\hat{u}$, where $\hat{u}_i=\langle \hat{f},g_i\rangle$ and  $A=(a_{ij})$ is the Gram matrix with entries $a_{ij}=\langle g_i,g_j\rangle$.  When $\hat{f}\ne f$, the estimate $\hat{w}$ may lie outside the simplex. Therefore, we use the direct estimate (\ref{weights0}) even when we do not have a closed form of $v^n(\theta)$ any more.
In our setting, the emission densities are also unknown, so we optimize a different objective function -- the pseudo-likelihood -- to estimate simultaneously both the densities and the weights.

A key feature of our estimation procedure is that the $L_2$ distance is used solely for estimating the mixture weights and not for the entire mixture density. There exists a large body of literature on distance-based estimation of mixture distributions, where the entire mixture model is estimated by minimizing a distance between a nonparametric (typically kernel-based) density estimate $\hat{f}_n(\cdot)$ and the model  density $\sum_i w_i g(\theta_i,\cdot)$; that is, the minimization is performed over both the weights and the component parameters. Commonly used distances include the $L_2$ and $L_1$ norms, as well as the Hellinger distance. Another option is to minimize the distance between empirical and theoretical distribution functions using the Wolfowitz, Cram\'{e}r-von Mises or Kolmogorov distance, see, e.g., \cite{cutler} and the references therein. These estimators are typically consistent, consistency follows from the continuity of the metric projection and the (relative) compactness of parameter space. In contrast, our estimator combines distance-based estimation with a likelihood-based approach, resulting in an objective function of a different nature. As a consequence, standard tools based on metric projection are not sufficient to establish consistency.
%--------

The current article is a direct follow-up of \cite{A1}, where the pseudo-likelihood method was introduced. The simulations in \cite{A1} demonstrate good behavior of the maximum pseudo-likelihood method -- it typically outperforms the $L_2$-based estimators and in some of the studied examples even beats the local maximizer of the likelihood function obtained with the EM algorithm. In particular, our method performs well when the number of mixture components $k$ is relatively large. This is understandable, since in the pseudo-likelihood approach the mixture weights are no longer treated as independent parameters, which reduces the number of parameters to be estimated. The larger the value of $k$, the greater the reduction. For a further discussion of the relationship with other estimation methods, an overview of the relevant literature, and simulations, we refer the reader to \cite{A1}.
%------------

The article is organized as follows. In Section \ref{sec:setting}, we introduce necessary notation and preliminaries, define the maximum pseudo-likelihood estimator, and state the main consistency theorem (Theorem \ref{thm}). In Section \ref{sec:about}, we give a brief overview of the proof and the guidelines for reading it. Sections 3-6 are devoted to the proof of Theorem \ref{thm}.

\subsection{Motivation for studying the pseudo-likelihood approach}
The need for a pseudo-likelihood arises from the well-known fact that the likelihood of Gaussian mixtures is unbounded. Several approaches have been proposed to address this issue, including restricting the parameter space, using sieves, penalized maximum likelihood estimation, Bayesian methods, profile likelihood, and others; see, e.g., \cite{ranalli, tanaka, tanaka-takemura} and the references therein. In \cite{chen}, consistency of the maximum likelihood estimator (MLE) is proved under the assumption that the variances of the mixture components are equal, a restriction that ensures the boundedness of the likelihood function. In \cite{tanaka-takemura}, consistency of the MLE is proved under the condition that all standard deviations are bounded below by $\exp[-n^d]$. In \cite{tanaka}, consistency is proved under the assumption that the ratio between the minimum and the maximum variance is bounded below by a sequence $b_n$ with $b_n\to 0$, or, more generally, when this ratio is suitably penalized. In this paper, we remove any restriction or penalties on the variances by replacing the likelihood function by the pseudo-likelihood function. The pseudo-likelihood function differs from the likelihood function only through the weights $v^n(\theta)$. In a sense, this represents the minimal modification of the likelihood function required to ensure boundedness without imposing any restrictions on the parameters. The price of this  modification is that existing consistency proofs and results are not directly applicable.

There are obviously many other ways to modify the likelihood function to
obtain an objective function with different properties. Since $v^n(\theta)$ is based on the kernel estimate $\hat{f}_n$, let us mention the so-called double-smoothed likelihood introduced in \cite{Seo17,SeoLindsey}.  Recall that the standard MLE minimizes the Kullback-Leibler divergence between the empirical measure and the assigned model. In the double-smoothed likelihood, both arguments of the Kullback-Leibler divergence -- the empirical measure and the assigned distribution -- are smoothed using the same kernel. The resulting function is bounded and, as shown in \cite{SeoLindsey13}, under rather general assumptions the maximizer of the double-smoothed likelihood (DS-MLE) is consistent. The proof is relatively straightforward and closely follows the classical proof of MLE consistency. To reduce the number of parameters, the weights in the double-smoothed likelihood function can be replaced with $v^n(\theta)$ (as in the  pseudo-likelihood function), and we conjecture that consistency still holds by standard arguments. An advantage of DS-MLE is that the kernel bandwidth can be fixed, i.e., chosen independently of the sample size $n$. In contrast, in our approach, the bandwidth must decrease sufficiently slowly to ensure the convergence $\hat{f}_n\to f$. On the other hand, a drawback of DS-MLE is its computational complexity. In \cite{SeoLindsey}, the authors propose using Monte-Carlo estimation, which is computationally demanding. In contrast, our pseudo-likelihood function can be easily computed and optimized using standard optimization tools, see \cite{A1}. From a theoretical perspective, we believe that keeping the pseudo-likelihood as close as possible to the likelihood helps preserve the desirable properties of the standard MLE. Simulations in \cite{A1} suggest that this may indeed be the case.

\subsection{Generalization beyond the i.i.d.~case}
Throughout the paper, we assume that $Y_1,Y_2,\ldots$ are i.i.d.~observations from a Gaussian mixture distribution. When inspecting the consistency proof, it becomes evident that the assumption of independent observations is used to apply the (uniform) strong law of large numbers, to ensure almost sure weak convergence of empirical measures and to guarantee almost sure convergence $\|\hat{f}_n\|_{\infty}\to \|f\|_{\infty}$, where $f$ denotes the true density. However, all these convergence results also hold in more general settings, suggesting that the consistency theorem may extend beyond the i.i.d.~mixture case to more general latent variable models. In particular, the following model is of interest. Let $X_1,X_2,\ldots$ be a stationary ergodic process taking values in $\{1,\ldots,k\}$ and let the observations $Y_1,Y_2,\ldots$ be as follows: 1) given $X_1,X_2,\ldots$, the observations are (conditionally) independent; 2) given $X_t=i$, the observation $Y_t$ has a Gaussian distribution with parameter $\t_i$. Such models are commonly used in many applications, where the latent $X$-process represents an underlying signal, and the observed $Y$-process models the signal corrupted by Gaussian noise. A classical example of such a model is a hidden Markov model, where the $X$-process is a Markov chain. In such a model, the weights $w^*_i$ are the probabilities $P(X_t=i)$ and the parameters $\theta_i$ are typically called the {\it emission parameters}. When estimating the emission parameters, the order of observations does not matter; thus, one can still use the pseudo-likelihood $L_n(\t)$ as defined above even when the model is not an i.i.d.~mixture model any more. Consistency in this context refers to the convergence of emission parameters as well as the marginal distribution of $X_t$. Due to the ergodicity, we conjecture that the consistency holds and the pseudo-likelihood method is justified for more general models than just i.i.d. mixtures. For maximum likelihood estimation, \cite{lindgren} proved that the maximum likelihood estimator of the parameters of a finite mixture distribution obtained under the assumption of independence (that is, ignoring the actual dependence structure) is consistent and asymptotically normally distributed when the regime process is an ergodic Markov chain. For the maximum spacing estimator, consistency of the estimator for the marginal parameters in hidden Markov models was established in \cite{kuljusMSP}.
%-------------
%
%The generalization beyond the i.i.d.~setting described above makes the pseudo-likelihood procedure %appealing for many applications, including signal processing.

\section{Consistency of pseudo-likelihood estimator}\label{sec:setting}
\subsection{Setting}
Let $\Theta_o=\big(\mathbb{R}\times (0,\infty)\big)^k$ be the parameter space. For every $\t=(\t_1,\ldots,\t_k)\in \Theta_o$, let $g(\theta_i,\cdot)$ stand for Gaussian density with parameter $\t_i=(\mu_i,\sigma_i)$. It may happen that some components coincide, so let $s(\t)\leq k$ be the number of different components.
  We shall identify all vectors $\theta$ with the same set of distinct components as a single parameter, thus $\Theta_o$ should be considered as the set of equivalence classes. For example, all permutations of $\t$ are equivalent. When $s(\t)=k$, then an equivalence class consists of only permutations, otherwise the class is larger. As a representative of an equivalence class, we consider $\t$ under a natural ordering of the parameters: $\mu_1 \leq \ldots \leq \mu_k$, and, in cases where some of the means are equal, the ordering is determined by the corresponding variances. When we discuss uniqueness or equality of parameter vectors, the natural ordering is assumed.

  We shall assume that the true parameter $\t^*$ is such that all components are different, i.e., $s(\t^*)=k$.  We denote the true density by $f$, i.e., $f(\cdot)=\sum_{i=1}^k w^*_ig(\t^*_i,\cdot)$, where $w^*_i>0$, $i=1,\ldots,k$.
 %i.e. $f$ is an honest $k$-component mixture.
 The requirement that $w^*_i > 0$ for every $i = 1, \ldots, k$ ensures that there exists no other parameter $\theta \neq \t^*$ and weights $w$ such that $\sum_i w_i g(\t_i, \cdot) = f(\cdot)$. This follows from the identifiability of Gaussian mixtures (see, e.g., \cite{teicher, schnatter}). The condition is also necessary for uniqueness, since when some weights of $w^*$ equal to zero, then  there exists $\theta\ne \t^*$ and weights $w$ such that
 $\sum_i w_i g(\t_i,\cdot)=f(\cdot)$ -- the true parameter would not be unique. For example, when  $w^*_1=0$, then $\theta$ could be taken as $(\t^*_2,\t^*_2,\t^*_3,\ldots,\t^*_k)\ne \t^*$ and $w$ could be taken as $(w^*_2/2,w^*_2/2,w^*_3,\ldots,w^*_k)$.
 %------------------------------------------------

 Since any $\t$ is an equivalence class, the convergence $\t^n\to \t$ is also a convergence between equivalence classes. For that we consider every class as a set and define the convergence between the parameters (classes) as the convergence between the sets in Hausdorff's sense. In our notation, when $\t=(\t_1,\ldots,\t_k)$ and $\t'=(\t'_1,\ldots,\t'_k)$ are two vectors in $ \Theta_o$, then the Hausdorff distance $h(\t,\t')$ is defined as
 $$h(\t,\t'):=\max\{\max_i \min_j \|\t_i-\t'_j\|,\max_i \min_j \|\t'_i-\t_j\|\}.$$ Clearly, $h(\t,\t')=0$ if and only if $\t$ and $\t'$ are in the same class.
 Note that the convergence $\t^n\to \t$ is equivalent to the condition that every subsequence $\t^{n'}$ has a further subsequence $\t^{n''}$, which converges component-wise to a representative of the equivalence class of $\t$.
 %Observe that the convergence $\t^n\to \t$ holds when there exist  orderings of the members of the sequence (not necessarily the natural ordering) $\t^n=(\t^n_1,\ldots,\t^n_k)$ such that $\t^n_i\to \t_i$ for every $i=1,\ldots,k$, and $(\t_1,\ldots,\t_k)$ is an element from the equivalence class of $\theta$. For example, the sequence $\t^n=(({(-1)^n 3\over n},2),({1\over n},1),({2\over n},1))$ converges component-wise to $((0,2),(0,1),(0,1))$, but is not component-wise convergent according to the natural ordering of elements of the sequence. As the equivalence class of the limit contains also $((0,1),(0,2),(0,2)$, we can not say that the convergence is equivalent to component-wise convergence to a reordering of any representative of the limiting equivalence class, but we can say that for a suitable reordering we have a component-wise convergence to a representative of the equivalence class of the limit.
 Therefore, we work mostly with point-wise convergent parameter sequences in this paper.

 It is important to understand that we can not replace our notion of convergence of sequences of parameter vectors with point-wise convergence of naturally ordered representatives: clearly we want to say that the sequence $(({(-1)^n\over n},1),({(-1)^{n+1} \over n},2))$  converges to $((0,1),(0,2))$, but the sequence of naturally ordered representatives does not converge point-wise and converging subsequences converge to different representatives of the equivalence class of the limiting vector.

 %All the convergences encountered in the paper are actually of this type, i.e., $\t^n\to \t$ implies the component-wise convergence.
 %If $s(\t)=k$, then obviously to every element of equivalence class (to every permutation) corresponds an ordering of $\t^n$ so that component-wise convergence holds. Complications in considering the convergence of the parameters is the price of considering parameters as the elements of a product space rather than sets or multisets. However, in what follows, we shall exploit the product space formalism, hence the price must be paid. It should also be noted that our conventions concerning the convergence do not affect our main result, because due to the assumption $s(\t^*)=k$, the convergence $\hat{\t}_n\to \t^*$ (consistency) could be interpreted as the Hausdorff convergence between sets or the component-wise convergence for any permutation of $\t^*$. Also observe that the assumptions $s(\t^*)=k$ and $w^*_i>0$ (for every $i$) imply that $\hat{\t}_n\to \t^*$ and $v^n\to w^*$ is equivalent to the weak convergence of probability measures: $\sum_i v^n_i \delta_{\hat{\t}^n_i}\Rightarrow \sum_i w^*_i \delta_{\t^*_i}$. The latter formalism is often employed in consistency literature, see \cite{chen}.}
  %-----------------------------------

Recall that $S_k$ stands for the $(k-1)$-dimensional simplex. For $r<k$, let $S_r$ denote the $(r-1)$-dimensional simplex. Recall the definition of weights $v^n(\theta)$ in (\ref{weights0}), where $\t\in \Theta_o$ and $\f$ is any density function in $L_2$.
Throughout the article, for any $w\in S_k$ and for any $k$-dimensional vector of Gaussian densities $g=(g_1,\ldots,g_k)$, we denote $wg:=\sum_{i=1}^kw_i g_i$. The following lemma guarantees that $v^n(\t)$ always exists and that the corresponding density $v^ng$ is unique.
%-----------------------------
\begin{lemma}\label{sisse}
Let $f,g_1,\ldots,g_k\in L_2$. Then there always exists at least one $v\in S_k$ such that
$\|f-vg\|=\inf_{w\in S_k}\|f-wg\|$. If $v_1$ and $v_2$ are two such vectors, then $v_1g=v_2g$.\end{lemma}
%-------------------------
\begin{proof}
We show that a minimizer exists. The existence of $v$ follows from the compactness of the simplex $S_k$ and the continuity of $w\mapsto \|f-wg\|$. If $v_1$ and $v_2$ are two different minimizers  such that $v_1g\ne v_2g$, the strict convexity of $L_2$ norm would imply that for any $\lambda\in (0,1)$,
$$\|f-(\lambda v_1+(1-\lambda)v_2)g\|=\|\lambda f-\lambda v_1g+ (1-\lambda) f-(1-\lambda) v_2g\|< \lambda\|f-v_1g\|+(1-\lambda)\|f-v_2g\|.$$
That contradicts the definition of $v_1$ and $v_2$.
\end{proof}
\\\\
Lemma \ref{sisse} ensures that for any $\t$, the solution of $\inf_{w\in S_k}\|f(\cdot)-wg(\t,\cdot)\|$, let it be  $v(\t)$, always exists, but  when $s(\theta)<k$, then it is not necessarily unique. However, given $\t$, the density $\sum_{i=1}^k v_i(\t) g(\t_i,\cdot)$ is always unique. Therefore, if $s(\t)=k$, then there are no other solutions $v'$ satisfying $\sum_{i=1}^k v_i(\t) g(\t_i,\cdot)=\sum_{i=1}^k v'_i g(\t_i,\cdot)$. This follows from the identifiability  of Gaussian mixtures --
when $g(\t_i,\cdot)\ne g(\t_j,\cdot)$ for all $i\ne j$, then
$w_1g=w_2g$ would imply $w_1=w_2$ (recall that we have fixed the ordering). Therefore, our assumption $s(\t^*)=k$ guarantees the uniqueness of $w^*$. The identifiability also implies that when $v$ is any minimizer of $\inf_{w\in S_k}\|f-wg\|$, then $\sigma_o=\max \{\sigma_i:v_i>0\}$ is unique, i.e., independent of the choice of a particular minimizer.
Let us remark that we are not aware of the closed form representation of $v^n$
except for the special case $k=2$ (see \cite{A1}, (6)).
%-----------------
\subsection{Consistency theorem}
Let $Y_1,Y_2,\ldots$ be a sequence of i.i.d.~random variables with true density $f(\cdot)=w^*g(\t^*,\cdot)$, where $\t^*\in \Theta_o$ and $w^*=(w^*_1,\ldots,w_k^*)$ are the corresponding strictly positive weights.  Given a nonparametric density estimator $\f$ and $g(\theta):=(g_1(\theta),\ldots, g_k(\theta))$, denote
\begin{equation}\label{ddd}
v^n(\t)=\arg\inf_{w\in S_k}\|\f-wg(\theta)\|,\quad  v(\t)=\arg\inf_{w\in S_k}\|f-wg(\theta)\|.\end{equation}
For every $\t\in \Theta_o$, define the log-pseudo-likelihood function $\ell_n(\t)$ as follows:
\begin{align*}
\ell_n(y,\t)&:=\ln \big(v^n(\t) g(\t,y)\big),\quad \ell(y,\t):=\ln \big(v(\t) g(\t,y)\big),\\
\ell_n(\t)&:={1\over n}\sum_{t=1}^n \ell_n(Y_t,\t),\quad \ell(\t):=E\ell(Y_1,\t).
\end{align*}
Sometimes, to stress the dependence of $\ell_n(\t)$ on $Y_1(\omega),\ldots,Y_n(\omega)$, we use the notation $\ell^{\omega}_n(\t)$. To keep the technique simpler, throughout the paper we ignore the cases where $\ell_n$ is unbounded (it happens with probability zero).  Recall that $v^n(\t) g(\t,\cdot)$ and $v(\t) g(\t,\cdot)$ are unique even when $v^n(\t)$ or $v(\t)$ are not. By our assumptions on $f$ (that is, $s(\t^*)=k$ and $w^*_i>0$ for every $i$), for any $w\in S_k$ and for any $\t\in \Theta_o$ such that $\t\ne \t^*$, it holds that $f(\cdot)\ne wg(\t,\cdot)$, and thus by Gibb's inequality
$$\int f(y) \ln (wg(\t,y))dy<\int f(y) \ln f(y) dy=\int f(y) \ln \big(w(\t^*)g(\t^*,y)\big)dy =\ell(\t^*).$$
Hence, for any $\t\ne \t^*$, it holds that $\ell(\t)<\ell(\t^*)$.
%-----------------
%\paragraph{Consistency of the maximum pseudo-likelihood estimator.}
%
We are interested in consistency of the pseudo-likelihood estimator $\tt$, where $\tt$ is for $\epsilon_n\searrow 0$ defined so that
\begin{equation}\label{mle}
\ell_n(\tt)\geq \sup_{\theta\in \Theta^o}\ell_n(\theta)-\epsilon_n.\quad    \end{equation}
Let $\tt=((\mu_{1,n},\sigma_{1,n}),\ldots,(\mu_{k,n},\sigma_{k,n})).$
Sometimes, to stress the dependence on $\omega$, we denote $\tt_{\omega}$. The main result of the article is the following consistency theorem.
%----------------
In the following, the convergence between functions means convergence in the $L_2$ norm if not stated otherwise.
%-------------------------
\begin{theorem}\label{thm} Assume that $\hat{f}_n \stackrel{a.s.}{\to} f$ (in $L_2$)  and $\exists C<\infty$ so  that $P(\|\hat{f}_n\|_{\infty}<C\,\, {\rm eventually})=1$. Then the  following convergences hold:
$$\tt\stackrel{a.s.}{\to} \t^*,\quad v^n(\tt)\stackrel{a.s.}{\to} w^*, \quad v^n(\tt)g(\tt,\cdot) \stackrel{a.s.}{\to} f(\cdot).$$
\end{theorem}
In the theorem, we do not assume that $\hat{f}_n$ is a kernel estimator, although in practice it is the most natural choice. Since we deal with the estimation of normal mixtures, it is natural to take $\hat{f}_n$ as the Gaussian kernel estimator. When the bandwidth of Gaussian kernel estimator tends to zero sufficiently slowly, then $\|\hat{f}_n-f\|_{\infty} \stackrel{a.s.}{\to}0$  \cite{silverman,zambom}, so that the assumptions on $\hat{f}_n$ are fulfilled.

\subsection{About the proof}\label{sec:about}

In one way or another, all consistency proofs rely on (relative) compactness of the parameter space. Perhaps the most direct and well-known example of this is the famous Wald consistency proof (see, e.g., \cite{chen, vaart}). Another common use of compactness is to establish the uniform convergence $\sup_{\theta} |\ell_n(\theta) - \ell(\theta)| \to 0$ almost surely, and then to use the fact that, on a compact space, uniform convergence implies the convergence of the maximizers (i.e., M-estimators). This is how we prove the consistency in the current
article. Although this approach is standard and widely used, applying it in the present setting involves several technical difficulties.

\paragraph{Unbounded means and vanishing or unbounded variances.} First, the compactification of the parameter space $\Theta_o$ includes zero and infinite variances, as well as infinite means. In the case of Gaussian distributions, zero variances are particularly problematic. Even in the case of i.i.d.~Gaussian random variables, one needs to apply the so-called Kiefer-Wolfowitz trick to handle vanishing variances when using the Wald consistency proof (see, e.g., (5.15) in \cite{vaart}).

To handle unbounded means and vanishing or unbounded variances,
we start by showing that at least one component of $\hat{\theta}_n$ is
such that its variance is bounded away from zero and above and its
mean is bounded as well. In particular, we show that there exist constants $0 < u < U < \infty$ and $N < \infty$ (depending only on the true density) such that, for all sufficiently large $n$, there exists -- with probability one -- a component $i(n)$ for which $|\mu^n_{i(n)}| < N$ and $u \leq \sigma^n_{i(n)} \leq U$.  This property is stated as Proposition \ref{uUNbound} and proved in Section \ref{sec:bound}. The proof uses some ideas from the proof of Lemma 3.1 in \cite{chen}. Proposition \ref{uUNbound} ensures that for every convergent subsequence $\hat{\theta}^{n'}\to \theta$, the limit $\theta$ also contains a component whose parameters are bounded as described above. It also allows us to reduce the parameter
space so that the parameters of at least one component are bounded as described above; this set is denoted
by $\Theta_o(u,U,N)$.

\paragraph{Uniform convergence of the criterion function.} The next step is to show that the uniform convergence of the criterion function holds over $\Theta_o(u,U,N)$, see (\ref{uniform}). Proving the uniform convergence is the main technical challenge of this article. First, observe that even for a fixed parameter $\theta$, we cannot directly apply the strong law of large numbers (SLLN) to deduce the convergence
$\ell_n(\theta) \stackrel{a.s.}{\to} \ell(\theta)$. This is because the weights
$v^n(\theta)$ depend on $\hat{f}_n$, and thus also on $\omega$. Thus, the standard SLLN does not apply in our case, and we must generalize it to accommodate the pseudo-likelihood setting as well. This generalization is formalized in Lemma \ref{dom}, which makes use of the Skorohod representation theorem.
Lemma \ref{dom} yields pointwise convergence $\ell_n(\theta) \stackrel{a.s.}{\to} \ell(\theta)$ for fixed $\theta$, but we also need the convergence of $\ell(\theta_n)$ for sequences $\theta_n \to \theta$, where the limit $\theta$ may involve zero or infinite variances and/or infinite means.

All possible limits beyond $\Theta_o$ require special treatment and, to some extent, novel techniques. These issues are addressed in Section \ref{sec:app}. The main result of that section is Proposition \ref{lemma}, which, together with Lemma \ref{dom}, leads to the uniform convergence result via Proposition \ref{prop1}. Once uniform convergence is established, the final consistency proof becomes standard,
%relying on Lemma \ref{prop2} and the concluding argument in Section \ref{sec:uniform}.
it is presented as the concluding argument of Section \ref{sec:uniform}.
\\\\
To recapitulate, from a broad perspective, our proof follows a standard path. However, almost every step along the way requires specific and largely novel techniques. The main difficulties arise from the fact that the estimates for the weights and parameters of the components are obtained by combining two different criterion functions ($L_2$ distance and likelihood). At the same time, we believe that this property -- applying two different criterion functions to obtain parameter estimates -- is one of the reasons behind the strong empirical performance of our estimator.
 %%%%%%%%%%%%%%%%%%%%%%%%%%%%%%%%%%%%%%%%%%%%%

%%%%%%%%%%%%%%%%%%
\section{Proof that $\tt$ belongs to $\Theta_o(u,U,N)$}\label{sec:bound}

Let $0<u<U<\infty$ and $0<N<\infty$ be fixed. Define
%\textcolor{blue}{
%$$\Delta(u,U,N):=\big((-\infty,N)\cup(N,\infty)\big)\times \big((0,u]\cup [U,\infty)\big),\quad     \Theta_o(u,U,N):=\Theta_o \setminus \big(  \Delta(u,U,N) \big)^k.$$}
%
\begin{equation} \label{uUNhulk}
\Delta(u,U,N):=\mathbb{R}\times \mathbb{R}^+ \setminus [-N,N] \times [u,U], \quad     \Theta_o(u,U,N):=\Theta_o \setminus \big(  \Delta(u,U,N) \big)^k.
\end{equation}
Thus, $\t\in \Theta_o(u,U,N)$ if and only if there exists $i$ such that $\sigma_i\in (u,U)$ and $|\mu_i|\leq N$. Let $\Theta(u,U,N)$ be the closure of $\Theta_o(u,U,N)$.
% If $\t\in \Theta(u,U,N)$, then $I_0(\t)\ne \emptyset$.

The following lemma was proved in \cite{A1}.
%------------------
\begin{lemma}\label{lem:toke} Let $f\in L_2$ and let $g_1,\ldots,g_k$ be Gaussian densities. Let $v=(v_1,\ldots,v_k)$ be any minimizer of (\ref{ddd}) for given $\theta$. Denote
$\sigma_o=\max \{\sigma_i:v_i>0\}$. Then
\begin{equation}\label{toke}
{v_i\over \sigma_i}\leq a+{b\over \sigma_o},\quad a=2\sqrt{\pi}\|f\|_{\infty},\quad b=2\sqrt{2}.
\end{equation}
\end{lemma}

Throughout this section, we assume that there exists $C<\infty$ such that $P(\|\f\|_{\infty}\leq C \,\,\rm{eventually})=1$.
In particular, this holds when $\|\f\|_{\infty} \stackrel{a.s.}{\to}\|f\|_{\infty}$. The latter holds when $\|\f-f\|_{\infty}\stackrel{a.s.}{\to}0$, let $\Omega_o$ denote the corresponding set. When such a $C$  exists, then by Lemma \ref{lem:toke}, we can assume without loss of generality the existence of universal constants $a>0$ and $b>0$ (depending on $\|f\|_{\infty}$), such that for every $\t\in \Theta_o$ and for every $\omega \in \Omega_o$,
\begin{equation}\label{toks}
\frac{v^n_i(\t)}{\sqrt{2\pi}\sigma_{i}}\leq \left(a+{b\over \sigma^n_{0}}\right),\quad \sigma^n_{0}:=\max\{\sigma_i:v^n_i(\t)>0\},
\end{equation}
%--------------------------------------------------------------
provided $n>n_o(\omega)$.
%------------

For every $u>0$, define the functions $U(u)$ and $N(u)$ as follows:
\begin{equation}\label{NU}
\frac{1}{\sqrt{2\pi} a U(u)}=e^{-\frac{1}{u}},\quad \exp\left(\frac{-N^2(u)}{8U^2(u)}\right)=e^{-\frac{1}{u}}.\end{equation}
Observe that both functions are decreasing in $u$ and $\lim_{u\to 0}N(u)=\lim_{u\to 0}U(u)=\infty$. Let $Y$ be a random variable with true distribution, consider
$$r_1(u):=1-P(|Y|>N(u)/2)-k\|f\|_{\infty}2\sqrt{2u}.$$
Then $r_1(u)\nearrow 1$ in the process $u\searrow 0$.
Thus, there exists $u_o$ such that $r_1(u)\geq 3/4$, whenever $u\leq u_o$. Define
$$r_2(u):=\sup_{z\in (0,u)}\left[\ln \left(a+{b\over z}\right)+\ln k - \frac{1}{2z}\right].$$
Since $\ln \left(a+{b\over z}\right)+\ln k - \frac{1}{2z} \to -\infty$ as $z\rightarrow 0$, there exists $u_W$ for every $-W>-\infty$ such that $r_2(u)\leq -W$, whenever $u\leq u_W$.
%-----
%
Take $-W<-\ln \big({\rm Var} Y\big)/2-2$ and fix the constants $u$, $U$, $N$ as follows:
\begin{equation}\label{constants}
0<u<\min\{u_W, u_o,U(u)\},\quad U:=U(u),\quad N:=N(u).
\end{equation}
Observe that the choice of $u$, $U$, $N$ depends solely on the true density $f$.
For every $0<u<U<\infty$ and $N$, define
\begin{align*}\bar{\Theta}(u,U,N):=\{&\t\in \Theta_o: \text{there exists a partition } \{J_1,J_2,J_3\} \text{ of } \{1,\ldots,k\} \\ &\text{ such that }\max_{i\in J_1}\sigma_i\leq u,\min_{i\in J_2}\sigma_i\geq U; \\
&\sigma_i\in (u,U), i\in J_3; \min_{i\in J_3} |\mu_i|>N\}.\end{align*}
Note that some of the sets in partition $\{J_1,J_2,J_3\}$ can be empty.
%----------
%-----------------------
\begin{proposition} \label{uUNbound} With $u$, $U$ and $N$ defined as in (\ref{constants}), the following holds:
\begin{equation}\label{maxmin2}
P(\lim\sup_n \{ \tt\in \bar{\Theta}(u,U,N)\})=P(\tt\in \bar{\Theta}(u,U,N)\,\,\rm{i.o.})=0,
\end{equation}
thus
$$P(\tt\in \Theta_o(u,U,N)\,\, {\rm eventually})=1.$$
\end{proposition}
%---------------
\begin{proof}
%--------------------------------------------------------
Let $u$, $U$, $N$ be defined as in (\ref{constants}). Fix $\omega\in \Omega_o$  and $n_o(\omega)$ so that (\ref{toks}) holds. Consider $\t\in \bar{\Theta}(u,U,N)$ and
let $\{J_1,J_2,J_3\}$ be the corresponding partition (depending  on $\t$). Take any $v^n(\t)=(v^n_1(\t),\ldots,v^n_k(\t))$ minimizing (\ref{weights0}), and denote $c^n_i(\t)=\frac{v^n_i(\t)}{\sqrt{2\pi}\sigma_{i}}$ and $c^n_0(\t)=a+{b\over \sigma^n_{0}}$. Then
$$\ln \big( v^n(\t)g(\t,y)\big)\leq \ln c^n_0(\t)+\ln k + \ln \max_i \left(\frac{c^n_i(\t)}{c^n_0(\t)}\exp\left(\frac{-(y-\mu_i)^2}{2\sigma_i^2}\right)\right).$$
%---------------
Define
$$A_{\t}=\{y\,:\,|y|\leq {N}/{2};\, |y-\mu_i|^2\geq 2\sigma_i,\ i\in J_1 \},$$
then due to (\ref{toks}) and by the choice of $N$ and $U$,
$$
\frac{c^n_i(\t)}{c^n_0(\t)}\exp\left(\frac{-(y-\mu_i)^2}{2\sigma_i^2}\right)\leq \begin{cases}
    1,& y\not\in A_\t,\\
    e^{-\frac{1}{\sigma_i}},& y\in A_\t,\ i\in J_1,\\
    \frac{1}{\sqrt{2\pi} a U}=e^{-{1\over u}},& y\in A_\t,\ i\in J_2,\\
    \exp(\frac{-N^2}{8U^2})=e^{-{1\over u}},& y\in A_\t,\ i\in J_3.
\end{cases}$$
Since for every $i\in J_1$ such that $v^n_i>0$,
$$e^{-\frac{1}{\sigma_i}}\leq e^{-\frac{1}{\min(\sigma^n_0,u)}},$$ and in the case $\sigma^n_0<u$ the sets $J_2$ and $J_3$ are empty or all weights are 0 for the components from those sets, we get
\begin{equation}\label{tokeA}
\ln \big( v^n(\t)g(\t,y)\big)\leq \ln \left(a+{b\over \min(u,\sigma^n_{0})}\right)+\ln k - \frac{1}{\min(u,\sigma^n_{0})} I_{\{y\in A_\t\}}(y).\end{equation}
To recapitulate: we have shown that for any $\omega\in \Omega_o$ and for any $\t\in \bar{\Theta}(u,U,N)$, the upper bound (\ref{tokeA}) holds, provided $n>n_o(\omega)$.
Due to our choice of $u$, $P(A_\t)\geq 3/4$, because
$$P(A_\t)\geq 1-P(|Y|>N/2)-|J_1|\|f\|_\infty 2\sqrt{2u}$$
$$\geq 1-P(|Y|>N/2)-k\|f\|_\infty 2\sqrt{2u}=r_1(u)\geq 3/4.$$
Since $A_\t$ consists of at most $k+1$ intervals, the Glivenko-Cantelli theorem gives that the following inequality holds  almost surely (let the corresponding set be $\Omega_{GC}$):
$$\inf_{\t\in \bar{\Theta}(u,U,N)} P_n(A_\t)\geq \frac{1}{2}\,\, \text{eventually}.$$
It follows by (\ref{tokeA}) that when $\omega \in \Omega_{GC}\cap \Omega_o$, then
\begin{equation}\label{ylem}\sup_{\t\in \bar{\Theta}(u,U,N)}\ell_n(\t)\leq \sup_{z\in (0,u)}\left[\ln \left(a+{b\over z}\right)+\ln k - \frac{1}{2z}\right]=r_2(u)\leq -W\,\, \text{eventually}.\end{equation}
%-----------------
On the other hand, by taking $\t^n_0=\big((\mu_n,S_n),\ldots,(\mu_n,S_n))$ (all components are equal), where
$\mu_n={1\over n}\sum_{t=1}^n Y_t$ is the sample mean and $S_n^2={1\over n}\sum_{t=1}^n(Y_t-\mu_n)^2 $ is the sample variance, we obtain
%------------
\begin{equation}\label{alu}
\ell_n(\t^n_0)=- \ln (\sqrt{2\pi})- \ln (S_n) - {1\over 2}\geq -\ln(S_n)-{2}.
\end{equation}
By SLLN, $\ln(S_n)\stackrel{a.s.}{\to} {1\over 2}\ln \big({\rm Var}(Y)\big)$, thus
\begin{equation}\label{all}
P\left(\ell_n(\tt)\geq -\frac{1}{2}\ln \big({\rm Var}(Y)\big)-2 \,\, {\rm eventually}\right)=1.
\end{equation}
Let $\Omega_V$ be the corresponding set. Recall that $-W< -\frac{1}{2}\ln \big({\rm Var}(Y)\big)-2$.
Let $\omega \in \Omega_{GC}\cap \Omega_V\cap \Omega_o$. If the corresponding $\tt_{\omega}$ is such that along a subsequence, $\hat{\t}_{\omega}^{n'}\in \bar{\Theta}(u,U,N)$, then by (\ref{ylem}), $\lim\sup_n \ell^{\omega}_n(\tt_{\omega})\leq -W$ -- a contradiction. Hence,
$$\lim\sup_n \{ \tt\in \bar{\Theta}(u,U,N)\}\subset \Omega^c_V\cup \Omega^c_{GC}\cup \Omega_o^c,$$
%This finishes the proof.
and $P(\tt\in \Theta_o(u,U,N)\,\, {\rm eventually})=1$ follows.
\end{proof}
%
%%%%%%%%%%%%%%%%%%%%%%%%%%%%%%%%%%%%%%%%%%%%%%%%%%%%%%%%%%%%%%
\section{Modification of SLLN}
%-------------
The following lemma generalizes SLLN. Note that for $h_n\equiv h$, (\ref{claim1}) reduces to the standard SLLN.
%--------------------------
\begin{lemma}\label{dom}Let $P$ be a probability measure, and let $h_n$ and $h$ be functions such that for $P$-a.e.~ $y$,  $y_n\to y$ implies $h_n(y_n)\to h(y)$. Let $Y_1,Y_2,\ldots$ be a sequence of i.i.d.~observations with distribution $P$, and let $H$ be a continuous function such that $EH(Y_1)=\int
H(y)P(dy)<\infty$. If $|h_n(y)|\leq H(y)$ for every $n$ and $y\in \mathbb{R}$, then
\begin{align}\label{claim1}
{1\over n}\sum_{t=1}^n h_n(Y_t)\stackrel{a.s.}{\to}  Eh(Y_1).
\end{align}
When $h\equiv -\infty$, and $h_n(y)\leq H(y)$ for every $n$ and $y\in \mathbb{R}$, then
\begin{align}\label{claim2}
{1\over n}\sum_{t=1}^n h_n(Y_t) \stackrel{a.s.}{\to} -\infty.
\end{align}
%---------
Moreover, the set, where (\ref{claim1}) and (\ref{claim2}) hold is $$\{P_n\Rightarrow P\}\cap\Big\{\int H(y)P_n(dy)\to \int
H(y)P(dy)\Big\},$$
where  $P_n$ is the empirical measure corresponding to
$Y_1,\ldots,Y_n$.
\end{lemma}
%------------
\begin{proof}
Let $P_n\Rightarrow P$ and $\int H(y)P_n(dy)\to \int
H(y)P(dy)$. Let now $Z_n$ and $Z$ be random variables such that $Z_n\sim P_n$, $Z\sim P$ and $Z_n \stackrel{a.s.}{\to} Z$. By the Skorohod representation theorem, such random variables exist.
Then $h_n(Z_n) \stackrel{a.s.}{\to} h(Z)$, $0\leq h_n(Z_n)+H(Z_n)\to H(Z)+h(Z)$, and $EH(Z_n)\to EH(Z)$, so by Fatou
$$E(h(Z)+H(Z))\leq \lim\inf_n (Eh_n(Z_n)+EH(Z_n)),$$
thus
$$Eh(Z)\leq \lim\inf_n E h_n(Z_n).$$
By Fatou, again,
\begin{equation}\label{fatou}
E\big(H(Z)-h(Z))\leq \lim\inf_n E(H(Z_n)-h_n(Z_n))=EH(Z)-\lim\sup_n
Eh_n(Z_n),\end{equation}
 so that $Eh_n(Z_n)\to Eh(Z)$. This establishes
(\ref{claim1}).

When $h\equiv\infty$, then by (\ref{fatou}), $$\infty \leq \lim\inf_n
E(H(Z_n)-h_n(Z_n))=EH(Z)-\lim\sup_n Eh_n(Z_n),$$ so that $\lim\sup_n
Eh_n(Z_n)\leq -\infty.$
Since $P_n\Rightarrow P$ almost surely, and by SLLN, \newline $\int H(y)P_n(dy) {\to} \int H(y)P(dy)$ almost surely, (\ref{claim1}) and (\ref{claim2}) hold with probability 1.
%----------------------
\end{proof}

\begin{corollary}\label{cory} Suppose that $\{h_n^{\omega}\}$ is a random sequence of measurable functions. Let $\Omega_o$ be the set such that, for every $\omega\in \Omega_o$, the following convergences hold:
$P_n^{\omega}\Rightarrow P$; for $P$-a.e. $y$ the convergence $y_n\to y$ implies $h^{\omega}_n(y_n)\to h(y)$; $|h_n^{\omega}(y)|\leq H^{\omega}(y)$, where $H^{\omega}$ is continuous, and $\int H^{\omega} dP_n^{\omega}\to \int H^{\omega} dP$.
Then $\forall \omega \in \Omega_o$,
\begin{align}\label{claimb}
{1\over n}\sum_{t=1}^n h^{\omega}_n(Y_t)\to Eh(Y_1).
\end{align}
\end{corollary}
%
%%%%%%%%%%%%%%%%%%%%%%%%%%%%%%%%%%%%%%%%%%%%%%%%%%%%%%%%%%%%%%%%%%%%
%\section{Approximation lemmas}\label{sec:app}
\section{Approximation of $f_n$ for a given set of normal components}\label{sec:app}
\noindent In this section, we shall consider $k$ sequences of normal
densities $g_i^n:=g(\mu_{i,n},\sigma_{i,n};\cdot)$ such that for
every $i\in\{1,\ldots,k \}$, the following limits exist:
\begin{align*}
\sigma_i=\lim_n\sigma_{i,n}\in [0,\infty],\quad \mu_i=\lim_n
\mu_{i,n}\in [-\infty,\infty].
\end{align*}
We also assume that for every $i$ and $j$, the limit $\lim_n
(\mu_{i,n}-\mu_{j,n})\in [-\infty,\infty]$ exists. This assumption is automatically fulfilled when $|\mu_i|<\infty
$ and $|\mu_j|<\infty$. It is important to realize that for any sequence of $k$ normal densities, one can choose a subsequence such that all these limits exist.

Let $f_n\to f$ be any convergent sequence. We consider an approximation of $f_n$ with $v^n g^n$, where the weights $v^n$ are defined as
 \begin{equation}\label{veiks}
 v^n:=\arg\inf_{w\in S_k} \|f_n-wg^n\|.\end{equation}
Recall that $v^n$ is not necessarily unique, but $v^n g^n$ is. In the main proposition of this section, we will study the convergence of $v^n g^n$.

%----------
Define the following partition of the set of component indexes $\{1,\ldots,k\}$:
\begin{align*}
I_0&:=\{i:\sigma_i\in (0,\infty),\,\, |\mu_i|<\infty\}, \quad r_0:=|I_0|;\\
I_1&:=\{i: \sigma_i\in (0,\infty),\,\, |\mu_i|=\infty\},\quad r_1:=|I_1|;\\
I_2&:=\{i: \sigma_i=\infty\},\,\, \quad r_2:=|I_2|;\\
I_3&:=\{i: \sigma_i=0\},\quad r_3:=|I_3|.
\end{align*}
Thus, $I_0$ consists of indexes such that variances converge to nondegenerate limits and means converge too; $I_1$ consists of indexes, where variances converge to nondegenerate limits, but the means diverge; $I_2$ consists of indexes, where variances
diverge; and $I_3$ is the set of indexes, where variances tend to 0.
%Clearly the sets partition $\{1,\ldots,k\}$.
For any $i\in I_0$, we denote $g_i:=g(\mu_i,\sigma_i;\cdot)$, thus $g_i^n\to g_i$ for every $i\in I_0$. Since
\begin{equation*}
\|g^n_i\|^2={1\over 2\sqrt{\pi}}\cdot{1\over
\sigma_{i,n}},\end{equation*} we see that
\begin{equation}\label{i2}
g_i^n\to 0,\quad \forall i\in I_2.
\end{equation}
Furthermore, note that for $i\in I_1$, the sequence of norms $\|g^n_i\|$ is bounded and $g^n_i(y)\to 0$ for every $y$, thus $g^n_i$ converges weakly to 0 in $L_2$ (denoted by $g^n_i\rightharpoonup 0$).
%--Kui vaja, siis viide nt https://projecteuclid.org/journals/real-analysis-exchange/volume-36/issue-2/Another-Proof-That-Lp-Bounded-Pointwise-Convergence-Implies-Weak-Convergence/rae/1321020515.pdf-----------
We shall denote
$$k^n_{ij}:=\langle g^n_i,g^n_j\rangle={1\over
\sqrt{2\pi(\sigma^2_{i,n}+\sigma^2_{j,n})}}\exp\left[-{(\mu_{i,n}-\mu_{j,n})^2\over
2(\sigma^2_{i,n}+\sigma^2_{j,n})}\right].$$ When $i,j\in I_0$, then
$k^n_{ij}\to k_{ij}:=\langle g_i,g_j\rangle$; when $i\in I_0$ and
$j\in I_1$, then  $k^n_{ij}\to 0$. Due to our additional assumption
about the existence of the limit $\lim_n(\mu_{i,n}-\mu_{j,n})$, clearly $k^n_{ij}\to
k_{ij}$ also when $i,j\in I_1$. Thus, the Gram matrix
$K^n=(k^n_{ij})_{i,j\in I_1}$ converges entry-wise to
$K=(k_{ij})_{i,j\in I_1}$.

Throughout this section, we shall assume that $r_0>0$, that is, $\exists i\in \{1,\ldots,k\}$ such that $g^n_i\to g_i$. Without loss of generality, we
denote this index by 1, so we shall assume that $1\in I_0$.

%--------------
Let us introduce some necessary notation. For any integer $r\geq 1$,
let $S_r$ be the $(r-1)$-dimensional simplex. % and for any $\alpha\in (0,1]$, let
% $$S_r(\alpha):=\Big\{(w_1,\ldots,w_r): w_i\geq 0, \sum_{i=1}^rw_i=\alpha\Big\},\quad
%T_r:=\Big\{(w_1,\ldots,w_r): w_i\geq 0, \sum_{i=1}^r w_i\leq
%1\Big\}.$$ Clearly $S_r(1)=S_r$ and $T_r=\cup_{\alpha\geq
%0}S_r(\alpha)$. For any vector $v\in S:=S_k$ let
%$S_r(v):=S_r(\sum_{i=1}^r v_i).$
Recall that for every vector
$w=(w_1,\ldots,w_k)\in S_k$ and densities $g=(g_1,\ldots,g_k)$, we
shall denote $wg:=\sum_i w_ig_i$. For any subset $I\subset
\{1,\ldots,k\}$ and for any vector $v\in S_k$, we denote the restrictions of $vg$ as follows:
$$v_Ig:=\sum_{i\in I}v_ig_i,\quad v_I:=(v_i)_{i\in I}.$$
%
%We shall also consider the more general objects
% $$\tv:=\arg\inf_{w\in T_k} \|f_n-wg^n\|.$$
%Given a subset $I\subset\{1,\ldots,k\}$, $|I|=r$, we define
%\begin{align*}
%v^{r,n}&:=\arg\inf_{w\in S_r} \|f_n-w_Ig^n\|
%,\quad
%\tvr:=\arg\inf_{w\in T_r} \|f_n-w_Ig^n\|
%.\end{align*}
%-----------------------
Recall that $g^n_i\to g_i$ for every $i\in I_0$. Let
\begin{align*}
&t: S_{r_0+r_1+r_2}\to \mathbb{R}^+,\quad
t(w):=\|f-w_{I_0}g\|^2+\sum_{i,j\in I_1}w_iw_jk_{ij},\\
&u:=\arg\min_{w\in S_{r_0+r_1+r_2}}t(w),\\
&t^*:=t(u)
%,\quad \tu= \arg\min_{w\in
%T_{r_0+r_1}}t(w),\quad \tilde{t}:=\min_{w\in
%T_{r_0+r_1}}t(w),\quad t^*:=\min_{w\in S_{r_0+r_1}}t(w)
.
\end{align*}
Note that when $r_2>0$, we have $u_i=0$ for $i\in I_1$, therefore $u_{I_0}$ can also be found by
$$u_{I_0}=\arg\inf_{w_i\geq 0,\,\sum_{i\in I_0}w_i\leq 1} \|f-w_{I_0}g\|.$$
The next lemma proves that the approximation $u_{I_0}g$ of the true density $f$ is unique.
%
%--------------------R.K.
%\begin{claim}\label{claim} The function $u_{I_0}g$ is unique. %and $\tu_{I_0}g$ are unique.
%\end{claim}
\begin{lemma}\label{claim} The function $u_{I_0}g$ is unique. %and $\tu_{I_0}g$ are unique.
\end{lemma}
\begin{proof} Note that %both
$S_{r_0+r_1+r_2}$ % and $T_{r_0+r_1}$ are
is a convex set and that $K=(k_{ij})_{i,j\in I_1}$ is a symmetric and positive semidefinite matrix.
Assume that $a$ and $b$ are two $(r_0+r_1+r_2)$-dimensional vectors. Then, using the notation $a\cdot b$ for the scalar product between two vectors, we have the following equality:
\begin{align*}
\frac{t(a)+t(b)}{2}-t\left(\frac{a+b}{2}\right)&=\frac{\|f-a_{I_0}g\|^2+\|f-b_{I_0}g\|^2-2\|f-\frac{a_{I_0}+b_{I_0}}{2}g\|^2}{2}\\
&\quad +\frac{2Ka_{I_1}\cdot a_{I_1}+ 2Kb_{I_1}\cdot b_{I_1}-K (a_{I_1}+b_{I_1})\cdot (a_{I_1}+b_{I_1})}{4}.
\end{align*}
As for any $x,y\in L_2$ we have
\[\|x+y\|^2+\|x-y\|^2=2(\|x\|^2+\|y\|^2),\]
we get for $x=f-a_{I_0}g$ and $y=f-b_{I_0}g$ the equality
\[\frac{\|f-a_{I_0}g\|^2+\|f-b_{I_0}g\|^2-2\|f-\frac{a_{I_0}+b_{I_0}}{2}g\|^2}{2}=\frac{1}{4}\|a_{I_0}g-b_{I_0}g\|^2.\]
Similarly, for any symmetric matrix $M$ and any vectors $v$ and $w$, we have
\[M(v+w)\cdot (v+w)+M(v-w)\cdot (v-w)=2(Mv\cdot v+Mw\cdot w),\]
therefore,
\[\frac{2Ka_{I_1}\cdot a_{I_1}+ 2Kb_{I_1}\cdot b_{I_1}- K (a_{I_1}+b_{I_1})\cdot (a_{I_1}+b_{I_1})}{4}=\frac{1}{4} K(a_{I_1}-b_{I_1})\cdot (a_{I_1}-b_{I_1}).\]
Thus, we have shown that
\[\frac{t(a)+t(b)}{2}-t\left(\frac{a+b}{2}\right)=\frac{1}{4}\left(\|a_{I_0}g-b_{I_0}g\|^2+ K(a_{I_1}-b_{I_1})\cdot (a_{I_1}-b_{I_1})\right).\]
Therefore, if $a$ and $b$ are two different minimum points of $t(\cdot)$ in a convex region, it follows from properties of $K$ that the second term is non-negative, and thus necessarily $a_{I_0}g=b_{I_0}g$.
\end{proof}
%------------R.K. end
\\\\
Recall an important bound from Lemma \ref{lem:toke}: for every $i\in \{1,\ldots,k\}$,
%------------------
\begin{equation}\label{toke}
{v^n_i\over \sigma_{i,n}}\leq a_n+{b\over \sigma_{0,n}},\quad a_n=2\sqrt{\pi}\|f_n\|_{\infty},\quad b=2\sqrt{2},\quad \sigma_{0,n}:=\max \{\sigma_{i,n}:v^n_i>0\} .
\end{equation}
In what follows, we shall assume that $\sup_n \|f_n\|_{\infty}<\infty$ and therefore, the constant $a_n$ in (\ref{toke}) can be chosen independently of $n$, so we shall use $a<\infty$ instead of $a_n$.
\\\\
%------------
The following auxiliary lemma will be needed in the proof of the main proposition of this section.
\begin{lemma}\label{lem:lisa1} Let $a^n\in S_k$ be an arbitrary sequence of weights such that the sequence
$\|f_n-a^ng^n\|$ is bounded above. Then $\sum_{i\in I_3}a_i^n\to
0$.\end{lemma}
\begin{proof}
By assumption, $\|f_n-a^ng^n\|$ is bounded above. The reverse triangular
inequality $\|f_n-a^ng^n\|\geq \|a^ng^n\|-\|f_n\|$ implies the
boundedness of $\|a^ng^n\|$. Now, because of nonnegativity of all terms, for every $n$ and $i$,
$$\|a^ng^n\|\geq a^n_i\|g_i^n\|.$$
Since $\forall i\in I_3$, $\|g_i^n\|\rightarrow \infty,$
it follows that $a^n_i\rightarrow 0$.
\end{proof}

%--------------------------------------------------------------------------------
\begin{corollary}\label{lem:lisa}  \textcolor{blue}{For any choice of $v^n$}, $\sum_{i\in I_3}v_i^n\to 0$.
\end{corollary}
\begin{proof} By assumption,
$g_1^n\to g_1$. Then $\|f_n-v^ng^n\|\leq  \|f_n-g_1^n\|\to
\|f-g_1\|$, so that $\|f_n-v^ng^n\|$ is bounded above. Thus, the assumptions of Lemma \ref{lem:lisa1} with $a^n=v^n$ are satisfied.\end{proof}
%
%---------------
%-----------------
\subsection{Convergence of the function $v^n g^n$}
Let $v^n$ be any vector of weights minimizing (\ref{veiks}). The main result of the present section is the following proposition.
\begin{proposition}\label{lemma} Let $f_n\to f$ and $\sup_n
\|f_n\|_{\infty}<\infty$. Let $y_n\to y$ be a convergent sequence such that
$y\not\in \{\mu_i:i\in I_3\}$. Then $v^ng^n(y_n)\to
u_{I_0}g(y)$ and
\begin{equation}\label{vaja}
v^{n}_{I_0}g^{n}\to {u}_{I_0}g.
\end{equation}
\end{proposition}
\begin{proof}
Denote
$$z_j^n=v^n_{I_{j}}g^n,\quad j \in \{0,1,2,3\}.$$
Clearly, $f_n,z_0^n$ and $z_1^n$ are bounded in $L_2$.
Using the notation $\langle \cdot,\cdot\rangle$ for the inner product in $L_2$, we can write
\begin{align*}\|f_n-v^ng^n\|^2&=\|f_n-z_0^n\|^2+\|z_1^n\|^2+\|z_2^n+z_3^n\|^2\\
&-2\langle f_n-z_0^n,z_1^n\rangle-2\langle f_n-z_0^n-z_1^n,z_2^n+z_3^n\rangle.
\end{align*}
Since $g_i^n\rightarrow 0$ and $|v_i|\leq 1$ $\forall i \in I_2$, we have $z_2^n\rightarrow 0$.  Next, we will show that $z_3^n\rightarrow 0$, then it is clear that the third and fifth terms above converge to 0 when $n \to \infty$.

%First, we will show that $v^n_i\to 0,\ i\in I_3$.
According to Corollary \ref{lem:lisa}, $v^n_i\to 0$ for every $i\in I_3$.
%--------
Denote $J:= I_0\cup I_1\cup I_2$ and
$\sigma_{0,n}:=\max_i \{\sigma_{i,n}: v^n_i>0\}.$ By Corollary
\ref{lem:lisa}, for every $n$ big enough, there
exists $i(n)\in J$ such that $\sigma_{0,n}=\sigma_{i(n),n}$. In
other words, $\sigma_{0,n}\geq \min_{i\in J}\sigma_{i,n}$. This, in turn, implies $\lim\inf_n \sigma_{0,n}\geq \lim_n  \min_{i\in
J}\sigma_{i,n}=\min_{i\in J}\sigma_i>0$. Since
$\sup_n\|f_n\|_{\infty}<\infty$, by (\ref{toke}) there exist constants $a$ and $b$ such that $v^n_i/\sigma_{i,n}\leq a+{b /\sigma_{0,n}}$.
 For every $i$,
\begin{equation}\label{l2norm}
\|g^n_i\|^2={1\over 2\sqrt{\pi}}\cdot{1\over
\sigma_{i,n}},\end{equation} so that when $\sigma_{i,n}\to 0$, then
 $$v_i^n\|g^n_i\|=(2\sqrt{\pi})^{-{1\over 2}}\cdot {v_i^n\over \sigma_{i,n}}\sqrt{\sigma_{i,n}}\to 0.$$
Thus $\|z_3^n\|\leq \sum_{i\in J^c}v^n_i\|g_i^n\|\to 0$, implying $z_3^n\to 0$.

Now we are ready to study the existence and properties of the limit of $v^ng^n$.
By the compactness of $S_k$, there exists a converging subsequence $v^{n'}\to w'$. Recall that $g^n_i\to g_i$ for every $i\in I_0$, thus $f_{n'}-z_0^{n'}\rightarrow f-w_{I_0}'g$.  Since $g_i^n\rightharpoonup 0\ \forall i\in I_1$, we have $z_1^n\rightharpoonup 0$, and therefore, $\langle f_{n'}-z_0^{n'},z_1^{n'}\rangle\to 0$. Finally,  $\|z_1^{n'}\|^2=K^{n'}v^{n'}_{I_1}\cdot v^{n'}_{I_1}\rightarrow Kw'_{I_1}\cdot w'_{I_1} $. Hence,
$$\|f_{n'}-v^{n'}g^{n'}\|^2\to t(w'_{J}).$$
On the other hand, if we define $$\tilde{u}_i=\begin{cases}
    u_i,& i\in J,\\
    0,&i\in I_3,
\end{cases}$$
then, by similar argument, we have
$$\|f_n-\tilde{u}g^n\|^2\rightarrow t(u)=t^*.$$
According to the definition of $v^n$, the inequality $\|f_{n}-\tilde{u}g^{n}\|^2\geq \|f_{n}-v^ng^{n}\|^2$ holds for every $n$, therefore $t(u)\geq t(w'_{J})$. According to the definition of $u$, this implies that $t(u)=t(w'_{J})$, and because of the uniqueness of $u_{I_0}g$, we have $w'_{I_0}g=u_{I_0}g$. Thus, for this subsequence, the convergence \eqref{vaja} holds. Since from every subsequence of the original sequence we can extract a subsequence which converges to the same limit, we have established  (\ref{vaja}).

Since for every $i\in I_0$, $\mu_{i,n}\to \mu_i$ and $\sigma_{i,n}\to \sigma_i>0$, for every convergent sequence $y_n\to y$, also $v^{n}_{I_0}g^{n}(y_n)\to {u}_{I_0}g(y)$. Since $y\in \mathbb{R}$, but for every $i\in I_1$, $|\mu_{i,n}|\to \infty$ and $\sigma_{i,n}\to \sigma_i \in (0,\infty)$, it follows that $g^n_i(y_n)\to 0$, thus $v^{n}_{I_1}g^{n}(y_n)\to 0$.  For every $i\in I_2$, $g_i^n(y_n)\to 0$. Finally, if $y\not\in \{\mu_i: i\in I_3\}$, then for every $i\in I_3$, $\exp[-{(y_n-\mu_{i,n})^2 / \sigma^2_{i,n}}]\to 0$. Since ${v_n / \sigma_{i,n}}\leq a+{b /\sigma_{0,n}}$, and $\sigma_{0,n}$ is bounded away from 0, we obtain $v^n_{I_3}g^n(y_n)\to 0$. Thus, $v^ng^n(y_n) \to {u}_{I_0}g(y)$.
\end{proof}
\begin{corollary}\label{lihtne} Let $f_n\to f$ and $\sup_n
\|f_n\|_{\infty}<\infty$. When $g_i^n\to g_i$ for every $i \in \{1,\ldots,k\}$, then $v^ng^n\to vg$, where
$v=\arg\min_{w\in S_k}\|f-wg\|$.\end{corollary}
\begin{proof} When $g_i^n\to g_i$ for every $i$, then
$I_0=\{1,\ldots,k\}$, $t(w)=\|f-wg\|^2$ and $u=v$. The convergence $v^ng^n\to vg$ follows from (\ref{vaja}).
\end{proof}
%
%
%-------------
\section{Uniform convergence of the criterion function} \label{sec:uniform}
%---------------------
\paragraph{Extending the pseudo-likelihood function.}
%--------------------------------------------
Recall the log-pseudo-likelihood function $\ell_n(\t)$. We enlarge $\Theta_o$, allowing some of the variances to be zero or infinite, and some of the means to be infinite. Thus, we define $\Theta:=\big([-\infty,\infty]\times [0,\infty]\big)^k$. We now extend the pseudo-likelihood to $\Theta$.
For every $\t\in \Theta$, let
\begin{align*}
I_0(\t)&:=\{i:\sigma_i\in (0,\infty),\, |\mu_i|<\infty\},\quad I_1(\t):=\{i:\sigma_i\in (0,\infty),\, |\mu_i|=\infty\},\\
I_2(\t)&:=\{i:\sigma_i=\infty\},\quad I_3(\t):=\{i:\sigma_i=0\}.\end{align*}
When $I_1(\t)\ne \emptyset$, we need the symmetric,  nonnegatively definite matrix $K=(k_{i,j})_{i,j\in I_1(\t)}$ defined in Section \ref{sec:app}, where $k_{i,i}={1\over 2\sigma_i\sqrt{\pi}}$.
For the elements of $\Theta$ with $|I_0(\t)|=r_0>0$, we extend the function $\ell(\t)$ as follows ($g_i(\cdot):=g(\t_i,\cdot), i\in I_0$). Recall that for any $w\in S_{r_0+r_1+r_2}$, $t(w):=\|f-w_{I_0}g\|^2+\sum_{i,j\in I_1}w_iw_jk_{ij}$, and $u:=\arg\min_{w\in S_{r_0+r_1+r_2}}t(w)$.
Let
%------------------
\begin{equation*}
\ell(\t,K):=E\ln \big({u}_{I_0}(\t)g(\t,Y_1)\big).
\end{equation*}
By Lemma \ref{claim}, the definition of $\ell(\t,K)$ is correct. Observe that when $I_1(\t)=\emptyset$, then $\ell(\t,K)$ is independent of $K$, and when $\sum_{i\in I_0}u_i=0$, then $\ell(\t,K)=-\infty$.
When $\theta\in \Theta_o$, then $I_0(\t)=\{1,\ldots,k\}$,
$t(w)=\|f-wg\|^2$, $I_2(\t)=\emptyset$ and $u(\t)=v(\t)$. Hence, $\ell(\t,K)$ extends $\ell(\theta)$. When $\sum_{i\in I_0}u_i\not=0$, we define
 $|{u}|=\sum_{i\in I_0}{u}_i$ and $\tilde{f}= {u}_{I_0}g /|{u}|$. So $\tilde{f}$ is a proper probability density function and if it is different from $f$, then by Gibb's inequality, $E\ln(\tilde{f}(Y_1))<\ell(\t^*)$. Now, whenever $\t\in \Theta$ is such that $\t\ne \t^*$ (in the sense of equivalence classes), then for every $K$ it holds that
\begin{equation}\label{pisi}
\ell(\t,K)= E\ln(\tilde{f}(Y_1))+\ln|{u}|<E\ln(\tilde{f}(Y_1))<\ell(\t^*).\end{equation}
%-------------------------
In the following proposition, $\t^n\to \t$ denotes the componentwise convergence $\mu_{i,n}\to \mu_i$ and $\sigma_{i,n}\to \sigma_i$, where we allow some limits $\mu_i$ to be infinite and $\sigma_i$ to be $0$ or $\infty$. Hence, $\t\in \Theta$. Moreover, we assume that $K_n = (\langle g^n_i, g^n_j \rangle)_{i,j \in I_1(\theta)}$ converges entrywise to the matrix $K$, denoted by $K_n \to K$. This convergence is required for Proposition \ref{lemma}.
%-----------------
\begin{proposition}\label{prop1} Assume that $\sup_n\|\f\|_{\infty}<\infty$ almost surely. Let $\t^n\in \Theta_o$ and $\t^n\to \t\in \Theta$, $|I_0(\t)|=r_0>0$ and $K_n\to K$. Then
 $\ell_n(\t^n)\stackrel{a.s.}{\to} \ell(\t,K)$ and $\ell(\t^n)\to \ell(\t,K)$. Moreover, the set of $\omega$'s where the almost sure convergence holds is independent of the sequence $\{\t^n\}$.
\end{proposition}
%---------
\begin{proof} Recall that $P_n$ is the empirical measure based on $Y_1,\ldots,Y_n$, and $P$ is the probability measure with density $f$. Let
\begin{align*}
\Omega_o:=&\{\omega: \f\to f\}\cap\{P_n\Rightarrow P\}\cap \{\sup_n\|\f\|_{\infty}<\infty\}\cap
        \{\int y^2 P_n(dy)\to \int y^2 P(dy)\}.\end{align*}
Take $\omega\in \Omega_o$ and denote $v^n=v^n(\t^n)$ (recall (\ref{ddd})), let $P_n^{\omega}$ be the corresponding empirical measures. Observe that $v^n$ depends on $\omega$. Take $f_n=\f$, then $f_n\to f$.

{\textbf{The upper bound of $\ln(v^ng^n(y))$}}. Consider a sequence $y_n\to y$, where $y\not \in \{\mu_i: i\in I_3(\t)\}$. By Proposition \ref{lemma},  $v^ng^n(y_n)\to {u}_{I_0}g(y)$. Apply Corollary \ref{cory} with
$$h^{\omega}_n(\cdot):=\ln \big(v^ng^n(\cdot)\big),\quad h(\cdot):=\ln \big({u}_{I_0}g(\cdot)\big).$$
It may happen that $\sum_{i\in I_0}{u}_i=0$, in which case, by (\ref{vaja}), we have $v^n_{I_0}g^n\to 0$. If this is the case, then (\ref{claim2}) of Lemma \ref{dom} establishes that $\ell_n(\t^n)\to -\infty$ (the function $H$ is constant, see the upper bound in (\ref{Hbound}) below). We now consider the case where $\|{u}_{I_0}g\|>0$. Let us show that there exists a continuous function $H$ such that $|h^{\omega}_n(y)|\leq H(y)$ and $\int H dP^{\omega}_n\to \int HdP$. By (\ref{vaja}), it holds that
$\|v^n_{I_0}g^n\|\to \|{u}_{I_0}g\|>0$, which implies that $\lim\inf_n \sum_{i\in I_0}v^n_i>0$.
%Indeed, if not, then $\lim\inf_n \|v^n_{I_0}g^n\|=0$, because
%$$\|v^n_{I_0}g^n\|\leq \sum_{i\in {I_0}}v^n_i\|g^n_i\|=(2\sqrt{\pi})^{-{1\over 2}}\sum_{i\in I_0}%{v^n_i\over \sqrt{\sigma_{i,n}}}\leq {\sum_{i\in I_0}v^n_i\over \min_{i\in I_0}\sqrt{\sigma_{i,n}}}$$
%and $\min_{i\in I_0}\sqrt{\sigma_{i,n}}\to \min_{i\in I_0}\sqrt{\sigma_i}>0$.
Thus, there exists $\alpha>0$ (depending on $\t$ and $K$, but independent of the choice of $v^n$) such that $\sum_{i\in I_0} v^n_i>\alpha$ eventually. This means that for every $n$ large enough, there exists $j\in I_0$ such that $v^n_j\geq {\alpha/ k}$, and therefore ${v^n_j/ \sigma_{j,n}}\geq {\alpha \over k\max_{i\in I_0}\sigma_{i,n}}$. Recall from (\ref{toke}) that $\sigma_{0,n}=\max\{\sigma_{i,n}:v^n_i>0\}$ and for $n$ large enough, $v_i^n/\sigma_{i,n}\leq a+b /\sigma_{0,n}$ (because for large $n$, $\|\f\|_{\infty}<\|f\|_{\infty}+1$). Thus, we have
\begin{align}\label{Hbound}
k\left(a+{b\over \sigma_{0,n}}\right)\geq  v^ng^n(y)\geq v^n_{I_0}g^n(y)\geq {1\over \sqrt{2\pi}} {\alpha\over k \max_{i\in I_0}\sigma_{i,n}}\exp\left[-{\max_{i\in I_0}(y-\mu_{i,n})^2\over \min_{i\in I_0}\sigma^2_{i,n}}\right].
\end{align}
Since $\sum_{i\in I_0}v^n_i>\alpha$ eventually, it holds that
$$\lim\inf_n \sigma_{0,n}\geq \lim_n  \min_{i\in I_0}\sigma_{i,n}=\min_{i\in I_0}\sigma_i>0.$$
Therefore, $\ln \big(v^n g^n(y)\big)$ is bounded above by a constant: $\ln(k(a+{b / \sigma_{0,n}}))\leq N_1$, where $N_1$ depends on $\min_{i\in I_0}\sigma_i$ and hence on $\t$.
%---

{\textbf{The lower bound of $\ln(v^ng^n(y))$}}. Observe that $\max_{i\in I_0}\sigma_{i,n}\to \max_{i\in I_0}\sigma_i<\infty$ and $\min_{i\in I_0}\sigma_{i,n}\to \min_{i\in I_0}\sigma_i<\infty$, thus there exist constants $N_2<\infty$ and $M_2<\infty$ depending on $\t$ and $K$ such that
\begin{align*}
-\ln \big(v^n g^n(y)\big) &\leq N_2+{\max_{j\in I_0}(y-\mu_{i,n})^2\over \min_{i\in I_0}\sigma^2_{i,n}}
 \leq N_2+{\sum_{i\in I_0}(y-\mu_{i,n})^2\over \min_{i\in I_0}\sigma^2_{i,n}} \\ &\leq N_2+M_2\sum_{i\in I_0}(y-\mu_{i,n})^2.
\end{align*}
Since for every $i\in I_0$ we have $\mu_{i,n}\to \mu_i\in \mathbb{R}$, there exist constants $A$ and $B$ depending on $\mu_i$ such that
$$\sum_{i\in I_0}(y^2-2y \mu_{i,n}+\mu^2_{i,n})\leq k y^2+A|y|+B.$$
Hence, by taking $H(y):=N_2+M_2(k y^2+A|y|+B)+ N_1$, we can see that $\int H P^{\omega}_n\to \int HdP$.
Since the assumptions of Corollary \ref{cory} are fulfilled, $\ell^{\omega}_n(\t^n)\to \ell(\t,K)$ follows. Since $P(\Omega_0)=1$, we obtain $\ell_n(\t^n)\stackrel{a.s.}{\to} \ell(\t,K)$.
%-------------------

For the proof of $\ell(\t^n)\to \ell(\t)$, take $f_n:=f$ and use Proposition \ref{lemma} to deduce that $\ln (c^ng^n(y_n))\to h(y)$, where $c^n=v(\t^n)$ and $h(y)=\ln({u}_{I_0}g(y))$. Observe that $c^n$ is independent of $\omega$. The convergence $$\ell(\t^n)=E\ln (c^ng^n(Y_1))\to Eh(Y_1)=\ell(\t)$$ now follows either by the dominated convergence theorem or by the argument above when taking $P_n=P$.
\end{proof}
%-------------------
%\paragraph{Uniform convergence.}
%------------------------------------
%
\\\\
Proposition \ref{prop1} implies the almost sure uniform convergence of the log-pseudo-likelihood over $\Theta_o(u,U,N)$.
\begin{corollary} \label{cor:pseudo:unif}
Let the assumptions of Proposition \ref{prop1} hold. Then
\begin{equation}\label{uniform}
P\Big(\sup_{\t\in \Theta_o(u,U,N)}|\ell_n(\t)-\ell(\t)|\to 0\Big)=1.
\end{equation}
\end{corollary}
\begin{proof}
Let $\Omega_o$ be the set with probability measure 1, where the convergences $\t^n\to \t$ and $K_n\to K$  entail $\ell_n(\t^n)\to \ell(\t,K)$, provided that $I_0(\t)\ne \emptyset$. Fix $\omega \in \Omega_o$.
When $\sup_{\t\in \Theta_o(u,U,N)}|\ell^{\omega}_n(\t)-\ell(\t)|\to 0$ fails, there exists a sequence $\t_{n}\in \Theta_o(u,U,N)$ and some $\epsilon_o>0$ such that $|\ell^{\omega}_n(\t^n)-\ell(\t^n)|>\epsilon_o$ for every $n$. It is easy to see that there exists a subsequence $\t^{n'}$ such that $\t^{n'}\to \t\in \Theta(u,U,N)$ and $K_{n'}\to K$.  Since $\omega\in \Omega_o$, Proposition \ref{prop1} establishes
$\ell^{\omega}_{n'}(\t^{n'})\to \ell(\t,K)$ and $\ell(\t^{n'})\to \ell(\t,K)$ -- a contradiction.
\end{proof}
%----------
%\paragraph{The uniform convergence implies the consistency.}
%
\\\\
We can now prove the main theorem of the article. The uniform convergence in (\ref{uniform}) implies the consistency of $\tt$, provided that $\tt$ eventually belongs to $\Theta_o(u,U,N)$.
%---------------
%\begin{lemma}\label{prop2} Suppose there exists $0<u<U<\infty$ and $0<N<\infty$ so that $\t^*\in
%\Theta_o(u,U,N)$ and  $P(\tt \in \Theta_o(u,U,N),\, {\rm eventually})=1$. Then $\tt\to \t^*$, a.s..
%\end{lemma}
%\begin{proof}
% Moved
%\end{proof}

\paragraph{The proof of Theorem \ref{thm}.}
%\\
%-----------------------
%\begin{proof}
% from Lemma

a) We start by proving that $\tt \stackrel{a.s.}{\to} \t^*$. By Proposition \ref{uUNbound}, there exist constants $u$, $U$, $N$ (depending solely on $f$) such that $P(\tt\in \Theta_o(u,U,N)\,\, {\rm eventually})=1$. These constants can be chosen so that $\t^*\in \Theta_o(u,U,N)$. Let $\Omega_o$ be the set where $\ell_n(\t^*)\to \ell(\t^*)$, $\tt \in \Theta_o(u,U,N)$ eventually, and the uniform convergence (\ref{uniform}) holds. Since all these events hold with probability one, clearly $P(\Omega_o)=1$. On this set, the following relationships hold:
\begin{align}\label{liminf}
\ell_n(\tt)&\geq \ell_n(\t^*)-\epsilon_n\to \ell(\t^*)\quad \Rightarrow \quad  \lim\inf_n \ell_n(\tt)\geq  \ell(\t^*).
\end{align}
Because of the uniform convergence (\ref{uniform}), $\lim\sup_n\ell_n(\tt)=\lim\sup_n \ell(\tt)\leq  \ell(\t^*)$, which together with (\ref{liminf}) implies $\ell_n(\tt)\to \ell(\t^*)$. From every subsequence of $\tt$, one can find a further subsequence $n'$ such that $K_{n'}\to K$ and $\t^{n'}\to \t\in \Theta_o(u,U,N)$. By Proposition \ref{prop1},
$\ell_{n'}(\hat{\theta}^{n'})\to \ell(\t,K)$. On the other hand, $\ell_{n'}(\hat{\theta}^{n'})\to \ell(\t^*)$, thus $\ell(\t,K)=\ell(\t^*)$. By (\ref{pisi}),  $\t=\t^*$. This implies $\tt\to \t^*$.

b) Since the convergence $\tt\to \t^*$ entails the convergence $g^n:=g(\tt,\cdot)\to g(\t^*,\cdot)=:g$,
and $w^*=\arg\min_{w\in S_k}\|f-wg\|$, the convergence $v^n(\tt)g^n\stackrel{a.s.}{\to}  w^*g$  follows from Corollary \ref{lihtne}. The convergence of weights follows from the uniqueness of Gaussian densities and our assumptions on $f$ (recall that $s(\t^*)=k$ and $w^*_i>0$ for every $i$), which imply that any convergent subsequence of $v^n(\tt)$ must have limit $w^*$. This concludes the proof.
%\end{proof}

\subsection*{Acknowledgements} This work was funded by the Estonian Research Council grant PRG865.

\end{document}